\documentclass[11pt, reqno]{article}
\usepackage{amssymb,amsthm,amsxtra,latexsym, amsmath}
\usepackage[utf8,nocaptions]{vietnam}
\usepackage{multirow}
\usepackage{multicol}
\usepackage[top=3.1cm, bottom=3.1cm, left=3cm, right=2.8cm] {geometry}
\usepackage{array}
\usepackage[symbol]{footmisc}
\usepackage{vwcol} % trước \begin{document}
\usepackage{indentfirst}

\newtheorem{theorem}{Theorem}[section]
\newtheorem{corollary}[theorem]{Corollary}
\newtheorem{lemma}[theorem]{Lemma}
\newtheorem{proposition}[theorem]{Proposition}
\newtheorem{example}[theorem]{Example}

\newcommand{\squeeze}{\spaceskip=0.15em\relax}

\DeclareMathOperator{\Ker}{Ker}

\DeclareMathOperator{\Ann}{Ann}

\DeclareMathOperator{\Ass}{Ass}

\makeatletter
\renewcommand*{\@biblabel}[1]{[\hfill#1].}
\makeatother

\usepackage{fancyhdr}
\makeatletter
\renewcommand{\@biblabel}[1]{[#1]}
\makeatother

\begin{document}

SEQUENTIAL $1$-COHEN-MACAULAYNESS FOR DIRECT SUM OF MODULES

	\fontsize{10pt}{11pt}{{\bf Nguyen Xuan Linh$^1$}, Hanoi University of Civil Engineering}
\footnotetext{\small Corresponding author: Email: linhnx@huce.edu.vn}\\
	\setcounter{section}{1}
		
\noindent {\bf 1. Introduction}

\ \ \ Throughout this paper, let $(R,\frak m )$ be a Noetherian local ring and  $M$ a finitely generated $R$-module. Let $H^0_{\frak m}(M)=D_t\subset \ldots \subset D_1\subset D_0=M$ be the dimension filtration of $M$, i.e. $D_{i+1}$ is the largest submodule of $M$ of dimension less than $\dim_R(D_i)$ for all $i\leq t-1$. Note that the dimension filtration of $M$ always exists uniquely. 

The class of Cohen-Macaulay modules is one of the central research objects of commutative algebra. Two important extensions of the notion of Cohen-Macaulay module are the notion of generalized Cohen-Macaulay module introduced by N. T. Cuong, P. Schenzel and N. V. Trung \cite{CST} and  the notion of sequentially Cohen-Macaulay module defined by R. P. Stanley \cite{St} in the graded setting  and by P. Schenzel \cite{Sch} in the local setting. In a natural way, the notion of sequentially generalized Cohen-Macaulay module was introduced in \cite{CN}.  We say that $M$ is {\it sequentially Cohen-Macaulay} (resp. {\it sequentially generalized Cohen-Macaulay}) if each quotient $D_i/D_{i+1}$ is Cohen-Macaulay (resp. generalized Cohen-Macaulay).

Let $H^i_{\frak m}(M)$ denote the $i$-th local cohomology module of $M$ with respect to $\frak m$.   Set  $$\frak a(M):=\frak a_0(M)~\frak a_1(M)\ldots \frak a_{d-1}(M),$$ where $\dim_R(M)=d$ and $\frak a_i(M)=\Ann_RH^i_{\frak m}(M)$ for  all $i$. Note that $M$ is generalized Cohen-Macaulay if and only if $\dim(R/\frak a(M))\leq 0.$ It suggests defining  the notion of  $1$-Cohen-Macaulay module and the notion of sequentially $1$-Cohen-Macaulay module as follows. We say that $M$ is {\it $1$-Cohen-Macaulay} if $\dim \big(R/\frak a(M)\big)\leq 1.$ We say that $M$ is  {\it sequentially $1$-Cohen-Macaulay} if each quotient $D_i/D_{i+1}$ is $1$-Cohen-Macaulay. 

Clearly, each Cohen-Macaulay module is  generalized Cohen-Macaulay; each generalized Cohen-Macaulay module is  $1$-Cohen-Macaulay. Similarly, each sequentially Cohen-Macaulay module is  sequentially generalized Cohen-Macaulay;
each sequentially generalized Cohen-Macaulay module is sequentially $1$-Cohen-Macaulay. The structures of these modules have attracted the interest of many mathematicians, for examples see [1]-[8].

From now on, let $(R,\frak m )$ be a Noetherian local ring and  $M_1,\ldots,M_n$  finitely generated $R$-modules. Set $M=M_1\oplus\ldots\oplus M_n.$ A characterization for $M$ being  sequentially Cohen-Macaulay was given in \cite[Proposition 3.2]{TPDA}. Moreover, if $M_1,\ldots,M_n$ are   sequentially generalized Cohen-Macaulay then so is $M$, see \cite[Proposition 4.5]{CN}.  The purpose of this paper is to characterize the sequential $1$-Cohen-Macaulayness of the direct sum $M$. 
The following theorem is the main result of this paper. 	

\medskip

\noindent{\bf Main theorem}.
 {\it $M=M_1 \oplus \ldots \oplus M_n$  is  sequentially $1$-Cohen-Macaulay if and only if  $M_i$ is sequentially  $1$-Cohen-Macaulay for all $i \leq n.$ }
 
\medskip

In the next section, we describe the largest submodule of the direct sum $M$. Then we prove the main theorem. An  example is given to clarify the results.
 
 \noindent {\bf 2. Methods} 
 
 In this paper, we employ inductive methods as well as the dimension filtration of a finitely generated module. We use mathematical induction to prove Lemma \ref{L22}, Lemma \ref{L25} and Theorem \ref{T26}. Furthermore, we provide a detailed analysis of the dimension filtration of a direct sum of modules relative to their largest submodules in Proposition \ref{P23} and Corollary \ref{C24}.
 \setcounter{section}{2}
 
\noindent {\bf 3. Main results}		

\setcounter{section}{3}

Firstly, we give a characterization for the direct sum $M$ being  $1$-Cohen-Macaulay. In order to do that, we need to recall the following lemma, see \textcolor{red}{\cite{NM}}.
\begin{lemma}\label{L21}
	Let  $0\to N_1\to N\to N_2\to 0$ be an exact sequence of finitely generated $R$-modules. Suppose $d=\dim_R(N)\geq 2.$ Then 
	\begin{itemize}
		\item[\rm{(a)}] If $\dim_R(N_1)\leq 1$ then $N$ is $1$-Cohen-Macaulay if and only if so is $N_2$. If $\dim_R(N_2)\leq 1$ then $N$ is $1$-Cohen-Macaulay if and only if so is $N_1$.
		\item[\rm{(b)}] Suppose that $\dim_R(N_1)=\dim_R(N_2)=d$. If $N_1, N_2$ are $1$-Cohen-Macaulay then so is $N$. If $N, N_2$ are $1$-Cohen-Macaulay then so is $N_1.$   
	\end{itemize}

\end{lemma}
Now we characterize the $1$-Cohen-Macaulayness of the direct sum $M.$ 

	\begin{lemma}\label{L22}
$M=M_1 \oplus \ldots \oplus M_n$  is $1$-Cohen-Macaulay if and only if, for each $i\leq n$, either $\dim_R(M_i)\leq 1$ or $M_i$ is  $1$-Cohen-Macaulay of dimension $\dim_R(M)$. 
\end{lemma}		
\begin{proof} The proof is by induction on $n.$ For $n=1$, it is trivial.  Let $n>1.$ Set $\dim_R(M)=d.$ If $d\leq 1$ then it is clear that $M$ and all $M_i$ are $1$-Cohen-Macaulay. Let $d\geq 2.$  Without loss of generality, we can assume that $\dim_R(M_1)=d.$ Set $N_1=M_1 \oplus \ldots \oplus M_{n-1}$ and $N_2=M_n.$ Then $\dim_R(N_1)=d.$ We consider two cases.
	
{\it Case 1}: $\dim_R(N_2)\leq 1.$ By Lemma \ref{L21}(a) and by induction, $M$ is $1$-Cohen-Macaulay if and only if $N_1$ is $1$-Cohen-Macaulay, if and only if,  for each $i\leq n-1$, either $\dim_R(M_i)\leq 1$ or $M_i$ is  $1$-Cohen-Macaulay of dimension $d$.

{\it Case 2}: $\dim_R(N_2)> 1.$ 
Suppose for each $i\leq n $,  either $\dim_R(M_i)\leq 1$ or $M_i$ is  $1$-Cohen-Macaulay of dimension $d$.  Then $\dim_R(N_2)=d.$ By induction, $N_1$ is $1$-Cohen-Macaulay. Therefore, $M$ is $1$-Cohen-Macaulay by Lemma \ref{L21}(b). 

Conversely, suppose that $M$ is $1$-Cohen-Macaulay. Since $N_2$ is a submodule of $M$, it follows by \textcolor{red}{\cite[Corollary 2.3(iii)]{NM}} that $\dim_R(N_2)=d$. For each $i\leq d$ we have  $$H^i_{\frak m}(M)\cong H^i_{\frak m}(N_1)\oplus H^i_{\frak m}(N_2).$$ Since $\dim(R/\frak a(M))\leq 1$, we have $\dim(R/\frak a(N_1))\leq 1$ and $\dim(R/\frak a(N_2))\leq 1$. Hence $N_1$ and $N_2$ are $1$-Cohen-Macaulay of dimension $d.$ By induction, for each $i\leq n-1 $,  either $\dim_R(M_i)\leq 1$ or $M_i$ is  $1$-Cohen-Macaulay of dimension $d$.

\end{proof}

  For each finitely generated $R$-module $N$, let $U_N(0)$  denote the largest submodule of $N$ of dimension less than $\dim_R (N).$ 

 Now we describe the  submodule $U_M(0)$ of $M.$
\begin{proposition} \label{P23}  Set $M=M_1\oplus M_2$ and $\dim_R(M)=d.$ Suppose $\dim_R(M_1)=d.$ Set $\dim_R(M_2)=d'.$ Then we have
$$U_M(0)=\begin{cases} U_{M_1}(0) \oplus U_{M_2}(0) & \text{ if } d'=d \\
U_{M_1}(0) \oplus M_2 & \text{ if } d'<d. 
\end{cases}$$
\end{proposition}
\begin{proof}
We consider two cases.

{\it Case 1}: $d'=d.$  We have $$\dim_R(U_{M_1}(0) \oplus U_{M_2}(0)) =\max \{\dim_R(U_{M_1}(0), \dim_R(U_{M_2}(0)) \} <d.$$ Hence we have $U_{M_1}(0) \oplus U_{M_2}(0) \subseteq U_{M}(0).$ Let $m \in U_{M}(0) \subseteq M_1\oplus M_2.$ Write $m=m_1+m_2$, where $m_1 \in M_1, m_2 \in M_2.$ We first prove that $\dim_R(Rm_1)<d.$ Suppose to contrary that $\dim_R(Rm_1)=d$. Then there exists $\frak p \in \Ass_R (Rm_1)$ such that $\dim (R/\frak p)=d.$ Then we write  $\frak p =\Ann(rm_1)$ for some $r\in R$. Note that $a(rm)=0$ if and only if $a(rm_1)=0$ and $a(rm_2)=0$ for all $a\in R.$ Hence $\frak p \supseteq \Ann(rm).$   Therefore, $$d=\dim(R/\frak p) \leq \dim_R(Rm) \leq \dim_R U_{M}(0) <d.$$ This gives a contradiction. Hence $\dim_R(Rm_1)<d$.  So, $m_1 \in U_{M_1}(0).$ 
By the same arguments, $m_2 \in U_{M_2}(0).$ Hence $m=m_1+m_2\in U_{M_1}(0) \oplus U_{M_2}(0).$ Therefore,   $U_M(0)=U_{M_1}(0) \oplus U_{M_2}(0).$

{\it Case 2}: $d'<d.$ As $\dim_R(U_{M_1}(0) \oplus M_2) <d,$ we have  $U_{M_1}(0) \oplus M_2 \subseteq U_{M}(0).$ Let $m \in U_{M}(0).$ Then  $m=m_1+m_2$ with $m_1 \in M_1, m_2 \in M_2.$ By the same arguments  to the case 1, we have $m_1 \in U_{M_1}(0).$ Hence $m\in U_{M_1}(0) \oplus M_2.$ Therefore,  $U_M(0)=U_{M_1}(0) \oplus M_2.$

\end{proof}

We have the following consequence of Proposition \ref{P23}.
\begin{corollary}\label{C24}  Set $M=M_1\oplus M_2$ and $\dim_R(M)=d.$ Suppose $\dim_R(M_1)=d, \dim_R(M_2)=d'.$
	The following statements are true.
\begin{itemize}
	\item [{\rm (a)}] If $d'=d$ then $M/U_M(0) \cong M_1/U_{M_1}(0) \oplus M_2/U_{M_2}(0).$ 
 
  \item[{\rm (b)}] If $d'<d$ then $M/U_M(0) \cong M_1/U_{M_1}(0).$
\end{itemize}
\end{corollary}
\begin{proof}
a) Let $\varphi: M \to M_1/U_{M_1}(0) \oplus M_2/U_{M_2}(0)$ be the epimorphism of $R$-modules given by $\varphi (m_1 + m_2)=\overline{m}_1+\overline{m}_2,$ where $m_i\in M_i$ and $\overline{m}_i=m_i+ U_{M_i}(0)$ for each $i\in\{1,2\}.$ Since $d'=d$ by the assumption, it follows by Proposition \ref{P23} that $\Ker \varphi =U_{M_1}(0) \oplus U_{M_2}(0)=U_M(0).$ Hence $M/U_M(0) \cong M_1/U_{M_1}(0) \oplus M_2/U_{M_2}(0).$ 

b) Suppose that $d'<d.$ Consider the homomorphism $\psi:M \to M_1/U_{M_1}(0) $ which is defined by    $\varphi (m_1 + m_2)=m_1+ U_{M_1}(0).$ It is clear that $\psi$ is an epimorphism of $R$-modules and $\Ker \psi =U_{M_1}(0)\oplus M_2.$ As $d'<d$, we get by Proposition \ref{P23} that $\Ker \psi =U_{M}(0)$ and so $M/U_M(0) \cong M_1/U_{M_1}(0).$ 

\end{proof}

\begin{lemma} \label{L25}
%Let $d\geq 2$ be an integer. 
Set $M=M_1\oplus M_2$ and $\dim_R(M)=d.$ Suppose $\dim_R(M_1)=d,$  $\dim_R(M_2)=d'.$ Then $M$  is  sequentially $1$-Cohen-Macaulay if and only if so are $M_1, M_2.$  
\end{lemma}
\begin{proof} 
We prove by induction on $d.$  If $d \leq 1$ then this is clear. Let $d>1.$ 
Assume that $M_1, M_2$  are sequentially  $1$-Cohen-Macaulay. Then $U_{M_1}(0), U_{M_2}(0)$ are sequentially $1$-Cohen-Macaulay.  
%we have  $U_M(0)= U_{M_1}(0) \oplus U_{M_2}(0)  \text{ if } d=d'$ or $U_M(0)= U_{M_1}(0) \oplus M_2  \text{ if } d' <d.$
As $\dim_R(U_{M}(0))<d$ we get by  Proposition \ref{P23} and by induction that $U_{M}(0)$ is sequentially $1$-Cohen-Macaulay. 

We consider two cases.

{\it Case 1}: $d'=d.$ By Corollary \ref{C24}, we have $M/U_M(0) \cong M_1/U_{M_1}(0) \oplus M_2/U_{M_2}(0).$ Since 
$M_1, M_2$  are sequentially  $1$-Cohen-Macaulay, $M_1/U_{M_1}(0)$ and $M_2/U_{M_2}(0)$ are $1$-Cohen-Macaulay of dimension $d$. By Lemma \ref{L22}, we have $M/U_M(0)$  is $1$-Cohen-Macaulay. Thus, $M$ is  sequentially $1$-Cohen-Macaulay. 

{\it Case 2}: $d'<d.$ By Corollary \ref{C24}, we have $M/U_M(0) \cong M_1/U_{M_1}(0)$. As $M_1/U_{M_1}(0)$ is $1$-Cohen-Macaulay, so is $M/U_{M}(0).$ Therefore, $M$ is  sequentially $1$-Cohen-Macaulay.

Conversely, suppose that $M$ is  sequentially $1$-Cohen-Macaulay. We consider two cases.

{\it Case 1}: $d'=d.$ We have $U_M(0)= U_{M_1}(0) \oplus U_{M_2}(0)$ by Proposition \ref{P23}. As 
$U_M(0)$ is  sequentially $1$-Cohen-Macaulay of dimension less than $d$, induction assumption implies that $U_{M_1}(0), U_{M_2}(0)$ are sequentially $1$-Cohen-Macaulay. By Corollary \ref{C24}, we have 
$$M/U_M(0) \cong M_1/U_{M_1}(0) \oplus M_2/U_{M_2}(0).$$   Since $M/U_M(0)$ is  $1$-Cohen-Macaulay, so are $M_1/U_{M_1}(0)$ and $M_2/U_{M_2}(0)$  by Lemma \ref{L22}. Thus, $M_1, M_2$  are sequentially  $1$-Cohen-Macaulay. 

{\it Case 2}: $d'<d.$ Since $U_M(0)= U_{M_1}(0) \oplus M_2$ by Proposition \ref{P23}, we get  by induction that   $U_{M_1}(0)$ and $M_2$ are sequentially $1$-Cohen-Macaulay. Since $M/U_M(0) \cong M_1/U_{M_1}(0)$ by Corollary \ref{C24} and $M/U_M(0)$ is  $1$-Cohen-Macaulay,  $M_1/U_{M_1}(0)$ is $1$-Cohen-Macaulay. Thus, $M_1, M_2$  are sequentially  $1$-Cohen-Macaulay. 
  \end{proof}
 Now we are ready to prove the main result of this paper.
 
\begin{theorem}\label{T26}
	$M=M_1 \oplus \ldots \oplus M_n$ is  sequentially $1$-Cohen-Macaulay if and only if  $M_i$ is sequentially  $1$-Cohen-Macaulay for all $i \leq n.$
\end{theorem}

\begin{proof}
%Let $d\geq 2$ be an integer
We prove by induction on $n.$  The case $n=1$ is clear. Let $n>1.$ Set $N_1=M_1$ and $N_2=M_2\oplus \ldots \oplus M_n.$ Then $M=N_1 \oplus N_2.$ By Lemma \ref{L25}, $M$ is sequentially $1$-Cohen-Macaulay if and only if so are $N_1, N_2$. By induction, $N_2$ is sequentially $1$-Cohen-Macaulay if and only if $M_i$ is sequentially  $1$-Cohen-Macaulay for all $i \geq 2.$ Therefore, $M$ is sequentially $1$-Cohen-Macaulay if and only if  so are $M_i$  for all $i \leq n.$
\end{proof}

We give an example to clarify  results.
\begin{example}{\rm
		Let $R= K[[x, y, z, t, w]]$ be the formal power series ring over a field $K$ and $\frak m=(x,y,z,t,w)$ the maximal ideal of $R$. Let $M_1=R/(x, z) \cap (y)\cap (t), M_2=R/(x, y)\cap (z, t).$  Set $M=M_1\oplus M_2.$ Then $\dim_R(M)=\dim_R(M_1)=4$ and $\dim_R(M_2)=3.$ We can check that $H^i_{\frak m}(M_1)=0$ for all $i\leq 2$ and $\dim(R/\frak a_3(M_1))=3>1.$  So $M_1$ is not $1$-Cohen-Macaulay. We have $H^i_{\frak m}(M_2)=0$ for all $i\leq 1$ and $\dim(R/\frak a_2(M_2)) =1$. So $M_2$ is $1$-Cohen-Macaulay. Note that $M$ is not $1$-Cohen-Macaulay by Lemma \ref{L22}.
		
		Note that $0\subset M_2$ is the dimension filtration of $M_2.$ Therefore, $M_2$ is sequentially $1$-Cohen-Macaulay. Set $D_1=(y)\cap(t)/(x,z)\cap (y)\cap (t).$ Then $0\subset  D_1\subset M_1$ is the dimension filtration of $M_1.$ We can check that $M_1/D_1$ is Cohen-Macaulay of dimension $4.$ Moreover, we note that $D_1\cong (x,z,y)\cap (x,z,t)/(x,z).$ So, we have the exact sequence $$0\to D_1\to R/(x,z)\to R/(x,z,y)\cap (x,z,t)\to 0.$$ It follows that $H^i_{\frak m}(D_1)=0$ for all $i\leq 2.$ Hence $D_1$ is Cohen-Macaulay of dimension $3.$ So, $M_1$ is sequentially Cohen-Macaulay of dimension $4.$ Therefore, $M$ is sequentially $1$-Cohen-Macaulay by Theorem \ref{T26}. By Proposition \ref{P23}, the dimension filtration of $M$ is $0\subset D_1\oplus M_2\subset M.$ We have $M/D_1\oplus M_2\cong M_1/D_1$ is Cohen-Macaulay of dimension $4.$ Since $\dim(R/\frak a_2(M_2))=1$, we have $\dim(R/\frak a_2(D_1\oplus M_2))=1.$ So, $M$ is not sequentially generalized  Cohen-Macaulay. Hence $M$ is not sequentially  Cohen-Macaulay.}
		\end{example}
		\noindent {\bf 4. Conclusion}
		
	In this paper, we describe the largest submodule of the direct sum $M$ of dimension less than $\dim_R(M)$. Then we establish a necessary and sufficient condition for a direct sum being $1$-Cohen-Macaulay. Our main contribution, Theorem \ref{T26}, characterizes the sequential $1$-Cohen-Macaulayness of a direct sum.  We also give  an example to illustrate the applications of  Proposition \ref{P23} and Theorem \ref{T26}. In future work, we plan to study related problems in combinatorics, such as the question of when a Stanley-Reisner ring is sequentially 1-Cohen-Macaulay.
 
	\end{document}